\documentclass{amsart}

\usepackage[all,cmtip]{xy}
\usepackage{todonotes}
\usepackage{comment}
\usepackage{enumitem}
\usepackage{graphicx}
\usepackage{tikz-cd}
\tikzset{commutative diagrams/row sep/normal = 7.5 ex}
\tikzset{commutative diagrams/column sep/normal = 8.5 ex}
\usepackage{multicol}
\usepackage{color}
\usepackage[colorlinks=true,linkcolor=blue, urlcolor=blue, citecolor=blue]{hyperref}
\usepackage[utf8]{inputenc} 

\usepackage{amsmath,amssymb,amsthm,bbm}

\usepackage{array, multirow} 
\usepackage{mathtools}

\newtheorem{theorem}{Theorem}[section]
\newtheorem{lemma}[theorem]{Lemma}
\newtheorem{corollary}[theorem]{Corollary}
\newtheorem{proposition}[theorem]{Proposition}
\newtheorem*{claim}{Claim}

\theoremstyle{definition}

\newtheorem{example}[theorem]{Example}

\theoremstyle{remark}

\newcommand{\PR}{\mathbb{P}}
\newcommand{\Aut}{\operatorname{Aut}}
\newcommand{\Pic}{\operatorname{Pic}}
\newcommand {\caL} {{\mathcal L}}
\newcommand {\caO} {{\mathcal O}}
\newcommand{\caM}{\mathcal{M}}

\newcommand{\cat}[1]{{\normalfont\textbf{#1}}}
\DeclareMathOperator {\bir} {{{bir}}}
\DeclareMathOperator {\cha} {{{char}}}
\DeclareMathOperator {\Chow} {{{Chow}}}
\DeclareMathOperator {\Mor} {{{Mor}}}
\DeclareMathOperator{\MMor}{\mathcal{M}\!{\it or}}
\DeclareMathOperator {\Spec} {{{Spec}}}
\DeclareMathOperator {\im} {{{im}}}

\DeclareMathOperator{\LLinSys}{\mathcal{L}\!{\it in} \mathcal{S}\!{\it ys}}
\DeclareMathOperator{\LinSys}{{LinSys}}
\DeclareMathOperator {\codim} {{{codim}}}

\newenvironment{subproof}{\begin{proof}[Proof of claim.]}{%
\end{proof}}

\newcommand*{\defeq}{\mathrel{\vcenter{\baselineskip0.5ex \lineskiplimit0pt
			\hbox{\scriptsize.}\hbox{\scriptsize.}}}%
	=}

\usepackage[style=alphabetic, minalphanames=3]{biblatex}

\AtEveryBibitem{%
	\clearfield{issn} 
	\clearfield{doi} 
	\clearfield{isbn} 
	
	\ifentrytype{online}{}{
		\clearfield{url}
	}
}

\newcommand{\citestacks}[1]{{\cite[\href{https://stacks.math.columbia.edu/tag/#1}{#1}]{stacks}}}

\begin{document}

\title{Smooth rational curves on rational surfaces}

\author{Lucas das Dores}
\address{
	IMPA, Estrada Dona Castorina 110, 22460-320, Rio de Janeiro, Brazil
	\textsc{\newline \indent 
		\href{https://orcid.org/0000-0002-8713-5308%
		}{\includegraphics[width=1em,height=1em]{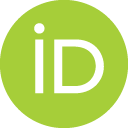} {\normalfont https://orcid.org/0000-0002-8713-5308}}
}}
\email{lucas.dores@impa.br}

\date{}

\begin{abstract}
	 Consider the scheme parametrizing non-constant morphisms from a fixed projective curve to a projective surface. There is a rational map between this scheme and the Chow variety of $1$-cycles on the surface. We prove that, if the curve is non-singular, then this rational map is a morphism. As a consequence, we obtain that, if the surface is rational and we fix a divisor class containing a non-singular rational curve, then the scheme parametrizing rational curves on this class is irreducible. Further, if the class has non-negative self-intersection, then the scheme of rational curves has expected dimension.
\end{abstract}

\maketitle
 
\section{Introduction}

There are many ways of viewing curves on an algebraic variety: as effective $1$-cycles, closed subschemes or morphisms from a fixed curve. Each of them gives rise to different parameter spaces of curves: Chow schemes, Hilbert schemes or the scheme of morphisms from a fixed curve, respectively. These parameter spaces are related to each other via canonical morphisms or rational maps; see \cite[Chapter I]{kollar-RCAV} for an overview.

The aim of this paper is to study the relation between the schemes of morphisms
from a fixed curve to a projective surface and the schemes parametrizing 1-cycles of
this surface.

Throughout the text we let $C$ be a projective curve and $X$ be a projective surface over an algebraically closed field $k$ of arbitrary characteristic. Let $\Mor_{>0}(C,X)$ be the scheme parametrizing non-constant morphisms from $C$ to $X$. Recall that Chow functors of $1$-cycles are introduced in \cite{kollar-RCAV}. These functors are represented by a scheme if $\cha k = 0$, but, in general, only coarsely represented by a scheme when $\cha k > 0$, see \cite[Theorems I.3.21 and I.4.13]{kollar-RCAV}. We denote this scheme $\Chow_{1}(X)$. Furthermore, there is a morphism
\begin{equation*}
\Theta: \Mor_{>0}^n(C,X) \rightarrow \Chow_{1}(X),
\end{equation*}
where $\Mor_{>0}^n(C,X)$ stands for the normalization of $\Mor_{>0}(C,X)$, see \cite[Corollary I.6.9]{kollar-RCAV}. This morphism is described by taking any $k$-point $[f]$ corresponding to a morphism $f:C \rightarrow X$ to the cycle associated to the proper pushforward $f_*[C]$. In particular, $\Theta$ defines a rational map from $\Mor_{>0}(C,X)$ to $\Chow_{1}(X)$.

On the other hand, since $X$ is a surface, $1$-cycles are just divisors. Hence, they can be parametrized by a functor taking each $k$-scheme $S$ to the set of relative effective Cartier divisors on $X \times S$ over $S$. This functor is representable in arbitrary characteristic and its representing scheme is the disjoint union of complete linear systems of divisors on $X$, which we denote $\LinSys(X)$, see Section \ref{daf}.

It is natural to ask whether we can obtain a morphism from $\Mor_{>0}(C,X)$ to $\LinSys(X)$ such as $\Theta$. Indeed, the first result of this paper tells us that when $C$ is nonsingular, then we have such a morphism without the need to use normalization. More precisely, we have the following.

\begin{theorem}
	\label{dac}
	Suppose $C$ is a nonsingular and irreducible projective curve and $X$ is a projective surface over $k$. Then there is a natural morphism
	\[
	\Xi: \Mor_{>0}(C,X) \rightarrow \LinSys(X)
	\]
	defined on $k$-points as $\Xi([f]) = f_*[C]$, which is invariant under the $\Aut(C)$-action
	\begin{align*}
	\Aut(C) \times \Mor_{>0}(C,X) & \longrightarrow \Mor_{>0}(C,X) \\
	([\alpha], [f]) & \longmapsto [f \circ \alpha]. 
	\end{align*}
\end{theorem}

Next, consider $C = \PR^1$. Let $\beta$ be a class in $\Pic X$ and $|\beta|$ be the corresponding complete linear system. Notice that the preimage of $|\beta|$ under $\Xi$ is a subscheme $\Mor(\PR^1,X, \beta) \subset \Mor_{>0}(\PR^1,X)$ parametrizing morphisms $f: \PR^1 \rightarrow X$ such that the divisor $f_*[\PR^1]$ belongs to the class $\beta$.

Let $\Mor_{\bir}(\PR^1,X,\beta) \subset \Mor(\PR^1,X,\beta)$ be the open subscheme parametrizing morphisms which are birational onto their images, and let $\Mor_{\overline{\bir}}(\PR^1, X, \beta)$ be the union of components of $\Mor(\PR^1, X, \beta)$ such that $\Mor_{\bir}(\PR^1,X,\beta) \cap \Mor(\PR^1,X,\beta) \neq \emptyset$, that is, the closure of $\Mor_{\bir}(\PR^1,X,\beta)$ in $\Mor(\PR^1,X,\beta)$. It is natural to ask under which conditions on $X$ and on the classes $\beta$ we can determine whether $\Mor(\PR^1, X, \beta)$ or $\Mor_{\overline{\bir}}(\PR^1, X, \beta)$ are irreducible. In \cite{testa-09}, it is proved that if $X$ is a Del Pezzo surface of degree greater or equal to $2$, then for any class $\beta$ in $\Pic X$, $\Mor_{\overline{\bir}}(\PR^1, X, \beta)$ is either irreducible or empty.

It is also natural to ask whether we can determine the dimension of the components of $\Mor(\PR^1, X, \beta)$. We have that for every irreducible component $M$ in $\Mor(\PR^1,X, \beta)$, there is a bound for its dimension
\begin{equation}
\label{aac}
\dim M \ge -K_X \cdot \beta + 2,
\end{equation}
where $K_X$ denotes the canonical class of $X$, see \cite[p.45]{debarre-HDAG}. We call this lower bound the \emph{expected dimension} of $M$.

The main result of this paper is a criterion to determine when $\Mor(\PR^1, X, \beta)$ is irreducible, and has expected dimension.

\begin{theorem}
	\label{aab}
	Let $X$ be a rational surface and let $\beta \in \Pic X$ be a class such that $|\beta|$ has a non-singular rational curve. Then, $\Mor(\PR^1, X, \beta)$ is irreducible and
	\begin{equation*}
		\dim \Mor(\PR^1, X, \beta) =
		\begin{cases}
		3, \text{ if } \beta^2 < 0, \\
		-K_X \cdot \beta + 2, \text{ if } \beta^2 \ge 0.
		\end{cases}
	\end{equation*}  
\end{theorem}

Under these conditions we see that $\Mor_{\overline{\bir}}(\PR^1, X,\beta) \neq \emptyset$. Hence, it is straightforward to deduce the following.

\begin{corollary}
	Let $X$ and $\beta$ satisfy the hypotheses of Theorem \ref{aab}. Then
	\[
	\Mor(\PR^1, X, \beta) = \Mor_{\overline{\bir}}(\PR^1, X, \beta).
	\]
\end{corollary}

\textbf{Structure of the paper.} In Section \ref{daf} we recall the definition of complete linear systems $|\beta|$ in terms of their functor of points, as well as the results needed to prove Theorem \ref{dac}. In Section \ref{caa} we prove Theorem \ref{aab} and, as an example, we give a complete list of the divisor classes satisfying the hypotheses of Theorem \ref{aab} on a smooth cubic surface in $\PR^3$.

\textbf{Acknowledgements.} I would like to express my gratitude to Eduardo Esteves and Vladimir Guletski\u{\i} for enlightening discussions. I am also thankful to Carolina Araujo, Thomas Eckl, and Roy Skjelnes for helpful suggestions.

This work was supported by CNPq, National Council for Scientific and Technological Development under the grant [159845/2019-0].

\section{Curves on surfaces}
\label{daf}

We recall that $\LinSys(X)$ is the scheme parametrizing effective divisors on $X$. More precisely, denote $\cat{Noe}/k$ to be the category of locally noetherian schemes over $k$, and $\cat{Set}$ to be the category of sets. Then, $\LinSys(X)$ is the scheme representing the functor 
\[
\LLinSys(X): (\cat{Noe}/k)^{op} \rightarrow \cat{Set}
\]
taking each $k$-scheme $f:S \rightarrow \Spec k$ to
\[
\LLinSys(X)(S) \defeq \left\{\text{Relative effective Cartier divisors of  } X_{S} \text{ over } S\right\},
\]
where $X_S= X \times S$. Furthermore, for each $\beta \in \Pic X$, let $\caL$ be an invertible sheaf in the class $\beta$. Hence, we can define the subfunctor 
\[
\LLinSys_{\beta}(X): (\cat{Noe}/k)^{op} \rightarrow \cat{Set}
\]
defined as
\[
\LLinSys_{\beta}(X)(S) \defeq \left\{\begin{array}{c}
\text{ Relative effective divisors } D \subset X_{S} \text{ over } S \\
\text{such that } \caO_{X_S}(D) \cong \caL_S \otimes f_S^*\caM \text{ for some } \\ \caM \text{ invertible on } S  \end{array}\right\},
\]
where $\caL_S$ and $f_S$ denote the base changes of $\caL$ and $f$. If $X$ is integral, then $\LLinSys_{\beta}(X)$ is represented by a projective space, see \cite[Theorem 9.3.13]{kleiman-05}. We call this projective space the \emph{complete linear system} of divisors of $\beta$ and denote it $|\beta|$. It follows that 
\[
\LinSys(X) = \coprod_{\beta \in \Pic X} |\beta|.
\]

We state two of the results used in the proof of Theorem \ref{dac}. In order to do so, we use the following notation: for any noetherian scheme $Y$ and any closed subscheme $W$ of $Y$ we denote $[W]$ to be the \emph{fundamental cycle} on $Y$ associated to it. We remark that both results come from much more general and detailed theories. The versions stated here have been simplified for our purposes. 

\begin{lemma}[{\cite{kleimanal-96}}]
	\label{daa}
	Let $S$ be a noetherian scheme, and let $f:Y_1 \rightarrow Y_2$ be a $S$-morphism of noetherian schemes. Suppose that
	\begin{enumerate}
		\item $f$ is finite;
		\item $f$ is a local complete intersection morphism;
		\item $\dim \caO_{Y_2,{f(p)}} - \dim \caO_{Y_1,p} = 1$ for all $p \in Y_1$ and;
		\item $f$ is curvilinear, \emph{i.e.} for each $p \in Y_1$, the rank of the fiber $\Omega^1_{Y_1/Y_2}(p)$ is at most one.
		
	\end{enumerate}
Then, the proper pushforward $f_*[Y_1]$ is the fundamental cycle $[W]$ associated to a relative effective Cartier divisor $W \hookrightarrow Y_2$ over $S$.
\end{lemma}
\begin{proof}
	This can be deduced from the aforementioned reference in the following way: from \cite[Proposition 2.10]{kleimanal-96} we have that $f$ is locally of flat dimension $1$, and by \cite[Corollary 2.5]{kleimanal-96} we have that $f_*[Y_1] = [W]$ for an effective divisor $W$ of $Y_2$. This divisor is defined by a zeroth Fitting ideal by \cite[Theorem 3.5]{kleimanal-96}, hence it follows by the local criterion of flatness that it is $S$-flat.
\end{proof}

\begin{lemma}[{\cite{suslinal-RCCS}}]
	\label{dab}
	Let $S$ be a connected noetherian scheme and $Y_1$ be a flat $S$-scheme. Let $f: Y_1 \rightarrow Y_2$ be a proper $S$-morphism of schemes of finite type and $S'$ be a noetherian $S$-scheme. Let $f_{S'}: Y_{1,S'} \rightarrow Y_{2,S'}$ be the base change of $f$. Suppose that $f_*[Y_1] = [W]$, where $W$ is a $S$-flat closed subscheme of $Y_2$. Then $(f_{S'})_*[Y_{1,{S'}}] = [W_{S'}]$.
\end{lemma}
\begin{proof}
	This is a particular case of \cite[Theorem 3.6.1]{suslinal-RCCS}.
\end{proof}

\begin{proof}[Proof of Theorem \ref{dac}]
	It suffices to define a morphism between the functors of points of $\Mor_{>0}(C,X)$ and $\LinSys(X)$. Notice that since $C$ is geometrically irreducible the functor of points of the former scheme evaluated at an irreducible noetherian $k$-scheme $S$ is given by
	\[
	\MMor_{>0}(C,X)(S) \defeq
	\left\{
	\begin{array}{c}
	S\text{-morphisms } g: C_S \rightarrow X_S \\
	\text{ which do not factor through } S
	\end{array}
	\right\}.
	\]
	We check that every morphism in this set satisfies the hypotheses of Lemma \ref{daa}.
	
	For each point $s \in S$, let $g_s: C_s \rightarrow X_s$ be the induced morphism between fibers and $D$ be the scheme-theoretic image of $g_s$, then either $D$ is a point or a curve. If $D$ was a point, then by the Rigidity Lemma \cite[Proposition 6.1, pg 115]{mumfordal-GIT}, we have that $C_S$ factors through $S$, which is a contradiction, therefore $D$ is a curve and $g_s$ factors through a surjective morphism $C_s \rightarrow D$. Since $C_s$ is irreducible this surjective morphism is finite, see \citestacks{0CCL}. In particular, for every $q \in X$ over $s$, $g^{-1}(q) \subset C_s$ is either empty or consists of finitely many points and since $g$ is proper we conclude $g$ is also a finite morphism. Also notice that $g$ factors through the smooth morphism $C_S \times_S X_S \rightarrow X_S$ via the graph morphism, therefore it is a local complete intersection.
	
	Further, since both $C_S$ and $X_S$ are flat over $S$, for any point $p \in C_S$ over $s \in S$, let $q = g(p) \in X_S$. We have that
	\[
	\dim \caO_{X_{S},{q}} - \dim \caO_{C_S,p} = \dim \caO_{X_s,{q}} - \dim \caO_{C_s,p},
	\]
	see \cite[Prop. III.9.5]{hartshorne-AG}. Hence, we have
	\[
	\dim \caO_{X_{S},{q}} - \dim \caO_{C_S,p} = \codim (\overline{\{q\}}, X_s) - \codim (\overline{\{p\}}, C_s).
	\]
	Notice that we can have two situations:
	\begin{enumerate}
		\item the closure of both $p$ and $q$ on the fibers are of maximal codimension;
		\item $p$ is the generic point of $C_s$ and $\overline{\{q\}} \subset X_s$ is of dimension $1$ (since if $\overline{\{q\}}$ was of dimension $0$ it would be a closed point, in this case $C_s$ would be contracted to a point by $g_s$ and, again by the Rigidity Lemma, $g$ would factor through $S$).
	\end{enumerate}
	In both situations we obtain $\dim \caO_{X_{S},{q}} - \dim \caO_{C_S,p} = 1$.
	
	Finally, notice that $\Omega_{C_S/X_S}(p) \cong \Omega_{C_s/X_s}(p)$ and we have the exact sequence
	\[
	g_s^*\Omega_{X_s/\kappa(s)} \rightarrow \Omega_{C_s/\kappa(s)} \rightarrow \Omega_{C_s/X_s} \rightarrow 0.
	\]
	It follows that we have a surjective map $\Omega_{C_s/\kappa(s)}(p) \twoheadrightarrow \Omega_{C_s/X_s}(p)$. Since $C_s$ is smooth over $\Spec \kappa(s)$, we have $\dim \Omega_{C_s/\kappa(s)}(p) = 1$, hence $\dim \Omega_{C_s/X_s}(p) \le 1$. In other words, $g$ is curvilinear.
	
	Thus, by Lemma \ref{daa} we have that the proper pushforward $g_*[C_S]$ is the fundamental cycle associated to a relative effective Cartier divisor $W \hookrightarrow X_S$ over $S$. Hence, for each irreducible noetherian $k$-scheme $S$ we have a well defined map
	\begin{equation}
	\Xi_S: \MMor_{>0}(C,X)(S) \rightarrow \LLinSys(X)(S)
	\end{equation}
	taking a morphism $g$ to $g_*[C_S]$. By Lemma \ref{dab}, we have that the maps $\Xi_S$ are natural on $S$. In other words, we have a natural transformation between the presheaves $\MMor_{>0}(C,X)$ and $\LLinSys(X)$ restricted to the category of irreducible noetherian schemes over $k$. 
	
	Since both presheaves are representable, they are Zariski sheaves. And since every locally noetherian scheme over $k$ is covered by irreducible noetherian schemes this natural transformation extends to the category $\cat{Noe}/k$. Yoneda Lemma implies this natural transformation corresponds uniquely to the morphism on the statement, which is $\Aut(C)$-invariant by definition.
\end{proof}

\section{Rational curves on rational surfaces}
\label{caa}

Before the proof of Theorem \ref{aab}, we recall two results. The first one is the theorem on dimension of fibers in the convenient presentation below. For a detailed proof see \cite[Proof of Proposition 5.5.1]{mustata-AG}.

\begin{lemma}[Dimension of fibers]
	\label{eaj}
	Let $f: Y_1 \rightarrow Y_2$ be a surjective morphism of algebraic varieties over $k$. Suppose that $Y_2$ is irreducible and that all (closed) fibers of $f$ are irreducible and of the same dimension $m$. Then:
	\begin{enumerate}
		\item there is a unique irreducible component $Y^0_1$ of $Y_1$ that dominates $Y_2$ and;
		\item every irreducible component $Z$ of $Y_1$ is a union of fibers of $f$ with \[
		\dim Z = \dim \overline{f(Z)} + m.
		\]
		In particular, $\dim Y^0_1 = \dim Y_2 + m$.
	\end{enumerate}
\end{lemma}

\begin{proposition}
	\label{dae}
	Let $\beta \in \Pic X$. Suppose that there exists a point $[C] \in |\beta|$ corresponding to a non-singular irreducible projective curve $C \subset X$. Then the general member of $|\beta|$ is irreducible and nonsingular.
	
	Furthermore, if $C$ is rational, then the general member is also rational.
\end{proposition}
\begin{proof}
	Since $|\beta|$ represents the functor $\LLinSys_{\beta}(X)$ there exists a universal relative effective divisor $\mathbb{D} \subset X \times |\beta|$, given by a flat morphism
	\[
	\varphi: \mathbb{D} \hookrightarrow X \times |\beta| \rightarrow |\beta|
	\]
	such that for every effective divisor $D \subset X$ in the class $\beta$ we have $\varphi^{-1}([D]) \cong D$. 
	
	Since $C$ is nonsingular, the morphism $\varphi$ has a nonsingular fiber and since it is flat there is an open subset $U \subset |\beta|$ such that  $\varphi|_{\varphi^{-1}(U)}$ is smooth, in particular, every fiber over $U$ is non-singular.
	
	Moreover, by Stein factorization \cite[Corollary III.11.5]{hartshorne-AG} we have that $\varphi$ factors as
	\[
	\varphi: \mathbb{D} \xrightarrow{\varphi_1} \mathbf{Spec} \, \varphi_*\caO_{\mathbb{D}} \xrightarrow{\varphi_2} |\beta|,
	\]
	where the fibers of $\varphi_1$ are connected and $\varphi_2$ is a finite morphism. Since $H^0(C, \caO_C) \cong k$, it follows from \cite[Corollaire 7.8.8]{grothendieck-EGAIII} that there is a neighbourhood $V \subset |\beta|$ of $[C]$ such that $\varphi_2|_{\varphi_2^{-1}(V)}$ is an isomorphism. We conclude that the fibres of $\varphi$ over the open subset $U \cap V \subset |\beta|$ are nonsingular and connected, and hence irreducible.
	
	For the last assertion, if we assume that $C$ is rational, then we have that its arithmetic genus is $p_a(C) =0$. Since $\varphi$ is flat, all of its fibers have the same Hilbert polynomial, and hence, the same arithmetic genus. In particular every fiber over $U \cap V$ has arithmetic genus $0$ and, thus, it is a rational curve.
\end{proof}

\begin{proof}[Proof of Theorem \ref{aab}]
	Denote $M_{\beta} \defeq \Mor(\PR^1, X, \beta)$ for simplicity and let $\Xi$ be the morphism defined in Theorem \ref{dac}. Then, we have $M_\beta \cong \Xi^{-1}(|\beta|)$. Notice that the image of the morphism
	\[
	\Xi_{\beta} \defeq \Xi|_{M_\beta}: M_{\beta} \rightarrow |\beta|
	\]
	consists of effective divisors supported on rational curves.
	
	The fiber of $\Xi_{\beta}$ over a point in $\im(\Xi_\beta)$ consists of an $\Aut(\PR^1)$-orbit of a point $[f]$ corresponding to a morphism $f: \PR^1 \rightarrow X$ such that $f_*\PR^1$ is in the class $\beta$. Since there are only finitely many automorphisms $\phi:\PR^1 \rightarrow \PR^1$ such that $\phi \circ f = f$ (see for instance \cite[Lemma 2.1.12]{kockal-IQC}), we have that the stabilizers $\Aut(\PR^1)_{[f]}$ are of dimension zero and since $\Aut(\PR^1)$ is irreducible we have that the $\Aut(\PR^1)$-orbit of $[f]$ is irreducible of dimension $\dim \Aut(\PR^1) = 3$.

	If $\beta^2 < 0$, then any effective divisor $D$ in $|\beta|$ is linearly equivalent to $C$ if and only if $D = C$. Hence, $|\beta|$ is a point and $\dim |\beta| = 0$, thus $M_{\beta}$ is a unique $\Aut(\PR^1)$-orbit of dimension $3$.
	
	If $\beta^2 \ge 0$, since $|\beta|$ contains a nonsingular irreducible curve $C$, Proposition \ref{dae} implies that $\im(\Xi_{\beta})$ contains a Zariski dense open subset $U \subset |\beta|$.
	
	Define $M_U \defeq \Xi_{\beta}^{-1}(U)$. By Lemma \ref{eaj} (applied to the reduction of $M_U$), there exists a unique irreducible component $M^0_{U} \subset M_{U}$ such that the image of $M^0_{U}$ dominates $U$ and
	\[
	\dim M^0_{U} = \dim U + 3 = \dim |\beta| + 3.
	\]
	
	Let $M^0_\beta$ be the irreducible component of $M_\beta$ whose underlying topological space is closure of $M^0_U$ in $M_{\beta}$, so that $\dim M^0_{\beta} = \dim M^0_U$.
	
	Let $[C] \in |\beta|$ be a non-singular rational curve, by Riemann-Roch and the adjunction formula we obtain
	\[
	\dim |\beta| = [C]^2 + 1 = -K_X \cdot \beta - 1.
	\]
	And thus $M^0_\beta$ has expected dimension $-K_X \cdot \beta + 2$. 
	
	\begin{claim}
		Suppose that $M_{\beta}' \subset M_{\beta}$ is an irreducible component such that $M'_\beta \neq M^0_\beta$. Then, $\Xi_\beta(M'_\beta)$ is contained in a proper closed subset of $|\beta|$.
	\end{claim}
	\begin{subproof}
	Suppose that $\Xi_\beta(M'_\beta)$ is not contained in a proper closed subset of $|\beta|$, in other words $\Xi_\beta(M'_\beta)$ is dense on $|\beta|$. In particular, $U \cap \Xi_\beta(M'_\beta)$ is dense in $U$ and thus $M'_{\beta} \cap M_U$ is an irreducible open subset of $M'_{\beta}$ dominating $U$. By definition of $M^0_U$ we have $M'_{\beta} \cap M_U \subseteq M_U^0$ and, thus, $M'_\beta \subset M^0_\beta$. Since $M'_\beta$ is an irreducible component we have $M'_\beta = M^0_\beta$.
	\end{subproof}
	
	Hence, if $M'_\beta$ is an irreducible component of $M_\beta$ distinct from $M^0_\beta$, by the second part of Lemma \ref{eaj} it follows that
	\[
	\dim M'_\beta = \dim \overline{\Xi_\beta(M'_{\beta})} + 3 < \dim |\beta| + 3 = -K_X \cdot \beta + 2
	\]
	which is impossible since we have the lower bound \eqref{aac}. Hence, $M_\beta = M^0_{\beta}$.
\end{proof}

\begin{example}
	\label{dah}
	Let $\sigma: X \rightarrow \PR^2$ be a blow-up of $\PR^2$ at $r \le 8$ points in general position $\{p_1, \dotsc, p_r\}$. Then $X$ is a Del Pezzo surface of degree $9-r$. The classes $\beta \in \Pic X$ such that $|\beta|$ contains nonsingular rational curves have been completely classified in \cite[Theorem 3.6]{gimiglianoal-13} for $r \le 7$. In other words, we have a complete list of classes satisfying the conditions of Theorem \ref{aab} for $X$. 
	
	On Table \ref{dai} we list such classes and the dimension of the irreducible $M_\beta \defeq \Mor(\PR^1, X, \beta)$ when $r = 6$, that is, when $X$ is isomorphic to a nonsingular cubic surface in $\PR^3$. In particular, we note that there exist linear systems $|\beta|$ with nonsingular rational curves for all values of $\beta^2$ greater or equal to $-1$. We refer the reader to the aforementioned reference for the complete list for all $r \le 7$.
	
	\begin{table}[h]
		\begin{center}
			\begin{tabular}{|c|c|c|c|}
				\hline
				\multirow{2}{3em}{\centering $\beta^2$} & \multirow{2}{2em}{\centering $d$} & \multirow{2}{8em}{\centering $(m_1, \dotsc,m_6)$ \\ {\small up to permutation}} & \multirow{2}{4em}{\centering $\dim M_\beta$} \\
				& & & \\
				\hline \hline
				\multirow{3}{3em}{\centering $-1$}  & $0$ & $(-1,0,0,0,0,0) $ &  \multirow{3}{3em}{\centering $3$}\\
				& $1$ & $(1,1,0,0,0,0)$ &  \\
				& $2$ & $(1,1,1,1,1,0)$ & \\
				\hline
				\multirow{3}{3em}{\centering $0$} & $1$ & $(1,0,0,0,0,0)$ & \multirow{3}{3em}{\centering $4$} \\
				& $2$ & $(1,1,1,1,0,0)$ & \\
				& $3$ & $(2,1,1,1,1,1)$ & \\
				\hline
				\multirow{9}{3em}{\centering $1 + 2t$} & $1+ t$ &$(t,0,0,0,0,0)$ & \multirow{9}{4em}{\centering $5 + 2t$} \\
				& $2 +t$   & $(1 + t,1,1,0,0,0)$ &  \\
				& $2 +2t$  & $(1+t,1+t,1+t,1+t,0,0)$ &  \\
				& $3 + t$  & $(2+t,1,1,1,1,0)$ & \\ 
				& $3 + 2t$ & $(2+t,1+t,1+t,1+t,1,0)$ & \\
				& $3 + 3t$ & $(2+2t,1+t,1+t,1+t,1+t,t)$ & \\
				& $4 + 2t$ & $(2+t,2+t,2+t,1+t,1,1)$ & \\
				& $4 + 3t$ & $(2+2t,2+t,2+t,1+t,1+t,1+t)$ & \\
				\hline
				\multirow{9}{3em}{\centering $2 + 2t$} & $2 + t$ & $(1+t,1,0,0,0,0)$ &  \multirow{9}{3em}{\centering $6 + 2t$} \\
				& $3 + t$ & $(2+t,1,1,1,0,0)$ & \\
				& $3 + 2t$ & $(2+t,1+t,1+t,1+t,0,0)$ & \\ 
				& $4 + 2t$ & $(2+t,2+t,2+t,1+t,1,0)$ & \\
				& $4 + t$ & $(3+t,1,1,1,1,1)$ & \\
				& $4 + 3t$ & $(3+2t,1+t,1+t,1+t,1+t,1+t)$ & \\
				& $5 + 2t$ & $(3+t,2+t,2+t,2+t,1,1)$ & \\
				& $5 + 3t$ & $(3+2t,2+t,2+t,2+t,1+t,1+t)$ & \\ 
				& $6 + 3t$ & $(3+2t,3+t,2+t,2+t,2+t,2+t)$ & \\
				\hline
				\multirow{4}{3em}{\centering $4$} & $4$ & $(2,2,2,0,0,0)$ &  \multirow{4}{3em}{\centering $8$} \\
				& $6$ & $(4,2,2,2,2,0)$ & \\
				& $8$ & $(4,4,4,2,2,2)$ & \\
				& $10$ & $(4,4,4,4,4,4)$ & \\
				\hline
				
			\end{tabular}
		\end{center}
		\caption{List of classes $\beta$ satisfying hypothesis of Theorem \ref{aab} on a smooth cubic surface in $\PR^3$.}
		\label{dai}
	\end{table}
	
	The information on Table \ref{dai} is organized as follows: for any $\beta \in \Pic X$ recall we can write $\beta = d \alpha - \sum_{i=1}^r m_i \mathcal{\varepsilon}_i$, where $\alpha$ is the pullback of the class of a line in $\PR^2$ and $\varepsilon_i$ is the class of the exceptional divisor $\sigma^{-1}(p_i)$. For each non-negative integer $t$, we list each $\beta$ in terms of $d \defeq d(t)$ and the tuples $(m_1, \dotsc, m_6) \defeq (m_1(t), \dotsc, m_6(t))$. Moreover, the classes are listed up to permutation of the $m_i$.
\end{example}

\printbibliography

\end{document}